\documentclass[a4paper,11pt]{article}
\usepackage{graphicx}
\usepackage{amssymb}
\usepackage{amsthm}
\usepackage{subcaption}
\usepackage{amsfonts}
\usepackage{amsmath}
\usepackage{hyperref}
\usepackage{xcolor}

\title{Weight-constrained universal sequences}
\author{Chris J Mitchell and Peter R Wild\\
Information Security Group, Royal Holloway, University of London}
\date{2nd November 2026 (v4)}

\theoremstyle{plain}
\newtheorem{lemma}{Lemma}[section]
\newtheorem{theorem}[lemma]{Theorem}
\newtheorem{corollary}[lemma]{Corollary}
\newtheorem{result}[lemma]{Result}

\theoremstyle{definition}
\newtheorem{definition}{Definition}[section]

\newtheorem{example}{Example}[section]

\theoremstyle{remark}
\newtheorem{remark}{Remark}[section]

\begin{document}

\maketitle

\begin{abstract}
In 2011/12, Ruskey, Sawada, Stevens and Williams showed that a binary universal sequence can be
constructed containing all $n$-tuples of Hamming weight either $r$ or $r+1$ for every possible $r$,
and no others.  In this paper we examine extensions of this result that require choosing a
weight-like function applying to $n$-tuples with symbols taken from an alphabet of arbitrary
(finite) size. We consider three possible such weight functions, and in each case establish an
analogous result to that of Ruskey, Sawada, Stevens and Williams.
\end{abstract}

\section{Introduction} \label{section:introduction}

The main focus of this paper is a special class of \emph{universal sequences} (or \emph{universal
cycles}), namely --- at least for the purposes of this paper --- infinite periodic sequences with
elements from a finite set in which every $n$-tuple of a certain type occurs exactly once.
Universal sequences of a range of types have been widely studied --- see, for example, Chung et
al.\ \cite{Chung92}. According to Hollmann and van Lint \cite{Hollmann97}, the term universal
sequence in this context dates back to a 1985 paper of Lempel and Cohn \cite{Lempel85}, although
they use the term in a much more restricted sense. The more general definition we use here is
consistent with the rather abstract definition given by Chroman et al. \cite{Chroman21}.

In this paper we are interested in sequences including all $n$-tuples for which some
weight-like measure is restricted.  Such sequences have been previously considered in the
binary case by Sawada et al.~\cite{Sawada11} and Ruskey et al.~\cite{Ruskey12}, who considered
binary $n$-tuples of restricted Hamming weight. In this paper we generalise the main existence
result of Sawada et al.\ and Ruskey et al.\ to finite alphabets of arbitrary size. In doing so
we look at three main examples of weight-like functions, one case of which was previously
considered by Gabri\'{c} et al. \cite{Gabric20}. Results of this type can also be derived from
the work of Sawada et al.~\cite{Sawada16} developed in a slightly different context, as we
discuss in Section~\ref{subsection:Sawada-Williams-Wong} below.

In particular we show that, for all three weight-functions, and for all integer pairs $r,s$ ($r<s$)
where $r$ and $s$ are possible weights for the relevant weight function, it is possible to
construct a universal sequence containing all $n$-tuples with weights in the range $[r,s]$. Precise
versions of these results can be found in Corollaries~\ref{corollary:H_characterisation},
\ref{corollary:S_characteristic}, and \ref{corollary:P_characteristic}.

\subsection{Background}

In this paper we consider periodic sequences $(s_i)$ with elements from $\mathbb{Z}_k$ for some
$k$, which we refer to as $k$-ary.  For a sequence $S = (s_i)$ and $n\geq1$ we write
$\mathbf{s}_n(i) = (s_i,s_{i+1},\ldots,s_{i+n-1})$, for an $n$-tuple corresponding to a
substring or factor of $n$ consecutive symbols occurring in the sequence at positions
$i,i+1,\ldots,i+n-1$. We say such an $n$-tuple appears in the sequence. Observe that in the
remainder of the paper we write $x^m$ as a shorthand for $m$ consecutive repetitions of the
symbol $x$.

\begin{definition}[\cite{Alhakim24a}]
A $k$-ary \emph{$n$-window sequence $S = (s_i)$} is a periodic sequence of elements from
$\mathbb{Z}_k$ ($k>1$, $n>1$) with the property that no $n$-tuple appears more than once in a
period of the sequence, i.e.\ with the property that if $\mathbf{s}_n(i) = \mathbf{s}_n(j)$ for
some $i,j$, then $i \equiv j \pmod m$ where $m$ is the period of the sequence.  Such a sequence is
also known as a \emph{cut-down de Bruijn sequence} (see, for example, Cameron et al.\
\cite{Cameron25}), and this is the term we use here.
\end{definition}

A $k$-ary de Bruijn sequence \emph{of span $n$} is then simply an $n$-window sequence in which
every $k$-ary $n$-tuple appears once in a period, i.e.\ an $n$-window sequence of maximal period.

Following Alhakim et al.\ \cite{Alhakim24a} and many previous authors, we also introduce the
$k$-ary de Bruijn digraph. For positive integers $n\geq1$ and $k\geq2$, let $\mathbb{Z}_k^n$ be
the set of all $k^n$ tuples of length $n$ with entries from the group $\mathbb{Z}_k$ of
residues modulo $k$. The order $n$ de Bruijn digraph, $B_k(n)$, is a directed graph with
$\mathbb{Z}^n_k$ as its vertex set in which, for any two vertices $\textbf{x} =
(x_0,x_1,\ldots,x_{n-1})$ and $\textbf{y} = (y_0,y_1,\ldots,y_{n-1})$, the pair
$(\textbf{x},\textbf{y})$ is an edge if and only if $y_i = x_{i+1}$ for every $i$ ($0\leq i<
n-1$). We label such an edge with the $(n+1)$-tuple $(x_0,x_1,\ldots,x_{n-1},y_{n-1})$.

Note that we have defined two ways of specifying an edge in $B_k(n)$, namely as either a pair
of vertices $(\textbf{a},\textbf{b})$, where $\textbf{a},\textbf{b}$ are $k$-ary $n$-tuples, or
as a single $k$-ary $(n+1)$-tuple $\textbf{x}$.  We refer to an edge in $B_k(n)$ and a $k$-ary
$(n+1)$-tuple interchangeably throughout.

\begin{definition}  \label{definition:Eulerian} An \emph{Eulerian digraph} is a
strongly connected digraph, i.e.\ one in which there is a directed path between any pair of
vertices, for which every vertex has in-degree equal to out-degree.  This latter property
characterises a \emph{balanced} digraph.
\end{definition}

The name derives from the fact that there exists an Eulerian circuit, i.e.\ a closed directed
path visiting every edge once, in a digraph if and only if the digraph is Eulerian --- see, for
example, Corollary 6.1 of Gibbons \cite{Gibbons85} or Knuth \cite[Section 2.3.4.2, Theorem
G]{Knuth97}. Moreover, there are simple and efficient algorithms for finding Eulerian circuits
--- see for example \cite[Figure 6.5]{Gibbons85} or \cite[Section 2.3.4.2, Theorem D]{Knuth97}.

There is then a simple correspondence between a directed Eulerian circuit in $B_k(n-1)$ and a de
Bruijn sequence of span $n$; this generalises in an obvious way so that a directed circuit (or
cycle) in $B_k(n-1)$ of length $m\leq k^n$ corresponds to a cut-down de Bruijn sequence of period
$m$.

We also need the following.

\begin{definition}[Necklace --- see, for example, \cite{Ruskey92}]
A $k$-ary \emph{necklace} within the de Bruijn graph $B_k(n-1)$, written as
$[a_0,a_1,\dots,a_{m-1}]$, where $m\mid n$, consists of all distinct $n$-tuples of the form
$(a_{\overline{i}},a_{\overline{i+1}},\dots,a_{\overline{i+n-1}})$, for $0\leq i<m$, i.e.\ all
cyclic shifts, where $\overline{s}=s\bmod m$.
\end{definition}

De Bruijn sequences have a range of real-world applications, ranging from position location on a
linear track or rotating shaft \cite{Burns92,Burns93,Peng23,Petriu98,Yamaguchi05}, to genome
assembly \cite{Compeau11}. As observed by Cameron et al.\ \cite{Cameron25}, `for some applications
it may be more convenient to produce a cycle of arbitrary length such that there are no repeated
strings', i.e.\ a cut-down de Bruijn sequence of a particular period that derives from the
specifics of the application.

A number of authors have considered the problem of constructing cut-down de Bruijn sequences with
particular properties. Jackson et al.\ \cite[Problem 477]{Jackson09} observed that a binary
$n$-window sequence exists for every possible period $m$ satisfying $n\leq m\leq2^n$. Sawada et
al.~\cite{Sawada11} and Ruskey et al.~\cite{Ruskey12} showed that a cut-down de Bruijn sequence can
be constructed containing all $n$-tuples with weights in the range $[a,b]$ as long as $0\leq
a<b\leq n$, and such sequences were further studied by Li et al.~\cite{Li23}. Later, Gabri\'{c} et
al.~\cite{Gabric20}, generalised this result to alphabets of an arbitrary size in the case where
weight is defined as the digit sum, and showed that all $n$-tuples bounded above (or below) by a
given weight can be arranged in a cut-down de Bruijn sequence. Even more recently, Gabri\'{c} et
al.~\cite{Gabric26} showed how to construct certain classes of universal sequences that can be
readily decoded.

Sequences containing only $n$-tuples with specified weights are a particular example of what
are sometimes known as `de Bruijn sequences with forbidden subsequences', as for example
examined in the binary case by Penne \cite{Penne10} and for arbitrary alphabets by Alhakim
\cite{Alhakim22}. Tan and Shallit \cite{Tan13}, studied the more general problem of the
existence of sequences containing only the members of a pre-specified set of $n$-tuples, where
they allow such $n$-tuples to occur multiple times in a period. These are all examples of
universal sequences, as introduced at the beginning of this paper.

\subsection{Cut down de Bruijn sequences in weight-restricted subgraphs}

As observed at the start of this paper, we are concerned with a particular class of cut-down de
Bruijn sequences, as previously considered by Sawada et al.~\cite{Sawada11} and Ruskey et al.
\cite{Ruskey12}. They showed that, if $n\geq1$ and $0\leq r<n$, the sub-graph of the de Bruijn
graph $B_2(n-1)$ consisting of all edges corresponding to $n$-tuples of Hamming weights $r$ and
$r+1$ is Eulerian, and hence there exists a cut down de Bruijn sequence which includes all
$n$-tuples of weights $r$ and $r+1$ (and no others).

In this paper we consider to what extent a similar result holds for the $k$-ary de Bruijn graph.
Of course, in this case there are a number of ways of computing the `weight' of an $n$-tuple, and
we therefore consider what happens for three different functions that assign a weight-like value to
an $n$-tuple.

\subsection{The Sawada-Williams-Wong result}  \label{subsection:Sawada-Williams-Wong}

Sawada et al.~\cite{Sawada16} considered sets $S$ of $k$-ary $n$-tuples satisfying two closure
properties:
\begin{itemize}
\item $S$ is closed under cyclic rotation of $n$-tuples;
\item $S$ is a $(k-1)$-suffix language, i.e.\ if $(a_0,a_1,\dots,a_{n-1})\in S$ then
    \[(a_0,a_1,\dots,a_{n-i-1},(k-1)^i)\in S\] for every $i\geq0$, that is $S$ is closed
    under replacement of the last $i$ entries of any $n$-tuple by a string of $i$ values
    equal to $k-1$.
\end{itemize}
They refer to the class of all such sets $S$ as $\mathcal{C}(n,k)$.  Using slightly different
language they show the following.

\begin{result}[Corollary 3 of \cite{Sawada16}]  \label{result:Sawada-Williams-Wong}
If $S\in\mathcal{C}(n,k)$ and $H$ is the subgraph of $B_k(n-1)$ whose edges correspond to the
$n$-tuples in $S$, then $H$ is Eulerian, and hence there exists a universal cycle containing
precisely the $n$-tuples in $S$.
\end{result}

In the discussions below we indicate how this relates to results for all three of the weight
functions we consider here.

\section{Weight functions for $k$-ary alphabets}

There are many ways of generalising the notion of Hamming weight of a binary $n$-tuple to the
$k$-ary case.  We briefly introduce some key possibilities here. The `standard' generalisation of
Hamming weight to the $k$-ary case is as follows.

\begin{definition}[Hamming weight]  \label{definition:Hamming_weight}
Suppose $k\geq2$ and $n\geq1$.  If $\mathbf{a}=(a_0,a_1,\dots,a_{n-1})$ is a $k$-ary $n$-tuple,
then the \emph{Hamming weight} $w_h(\mathbf{a})$ of $\mathbf{a}$ is equal to the number of values
$i$ for which $a_i\neq0$.
\end{definition}

It should be clear that in the binary case, i.e.\ where $k=2$, the Hamming weight corresponds to
the number of ones in the $n$-tuple, i.e.\ the expected definition.  An alternative measure for
$n$-tuples, which also corresponds to Hamming weight in the case $k=2$, is as follows.

\begin{definition}[Digit sum]  \label{definition:digit_sum}
Suppose $k\geq2$, $n\geq1$ and $\mathbf{a}=(a_0,a_1,\dots,a_{n-1})$ is a $k$-ary $n$-tuple. Then
the \emph{digit sum} (or \emph{symbol sum}) of $\mathbf{a}$ is defined as
\[ w_s(\mathbf{a}) = \sum_{i=0}^{n-1}a_i \]
where we compute the sum in $\mathbb{Z}$ and treat the values $a_i$ as integers in the range
$[0,k-1]$.
\end{definition}

\begin{remark}
Note that the digit sum is referred to simply as the \emph{weight} of an $n$-tuple in
\cite[Definition 11]{Mitchell25a}. It should be clear that the digit sum is also a natural
generalisation of Hamming weight for the binary case.
\end{remark}

Yet another alternative way of assigning a numeric value to an $n$-tuple is as follows.

\begin{definition}[\cite{Mitchell25a}, Definition 10] \label{definition:pseudoweight}
Suppose $k\geq2$, $n\geq1$ and $\mathbf{a}=(a_0,a_1,\dots,a_{n-1})$ is a $k$-ary $n$-tuple. Define
the function $f:\mathbb{Z}_k\rightarrow\mathbb{Q}$ as follows: for any $a\in\mathbb{Z}_k$ treat $a$
as an integer in the range $[0,k-1]$ and set $f(a)=a$ if $a\neq0$ and $f(a)=k/2$ if $a=0$. Then the
\emph{pseudoweight} of $\mathbf{a}$ is defined to be the sum
\[ w_p(\mathbf{a}) = \sum_{i=0}^{n-1}f(a_i) \]
where the sum is computed in $\mathbb{Q}$.
\end{definition}

This apparently curious notion was introduced to help generate special classes of cut-down de
Bruijn sequences known as orientable sequences and negative orientable sequences (see, for example,
\cite{Mitchell26a}).

In the remainder of this paper we consider to what extent generalised versions of the Sawada et
al.~\cite{Sawada11} and Ruskey et al.~\cite{Ruskey12} result hold for these three weight
functions. We first need the following preliminary definition and result that simplify
arguments for all three weight functions we examine.

\begin{definition}
Suppose $k\geq2$, $n\geq1$ and $w$ is function mapping $k$-ary $n$-tuples to $\mathbb{Q}$.  Then
$w$ is said to be \emph{permutation-agnostic} if, for any $k$-ary $n$-tuple $\mathbf{a}$,
$w(\mathbf{a})=w(p(\mathbf{a}))$ for every permutation $p$ acting on the $n$ entries in
$\mathbf{a}$.
\end{definition}

The three weight functions $w_h$, $w_s$ and $w_p$ are all permutation-agnostic since they are all
defined as a function of the elements of an $n$-tuple in a way that does not depend on the
ordering.  We observe the following result applying to all such weight functions.

\begin{lemma}  \label{lemma:one_change}  \label{lemma:balanced}
Suppose $k\geq2$, $n\geq2$, $r,s\in \mathbb{Q}$ ($r\leq s$), and $w$ is permutation-agnostic
function mapping $k$-ary $n$-tuples to $\mathbb{Q}$.  Let $W_k^{r,s}(n-1)$ be the subgraph of
$B_k(n-1)$ consisting of the edges corresponding to $n$-tuples $\mathbf{a}$ satisfying $r\leq
w(\mathbf{a})\leq s$. Then:
\begin{itemize}
\item[(i)] if $\mathbf{a}$ and $\mathbf{b}$ are edges ($n$-tuples) in $W_k^{r,s}(n-1)$ that
    satisfy $w(\mathbf{a})=r$ and $w(\mathbf{b})=s$ and that differ in a single position, then
    there are directed paths from $\mathbf{a}$ to $\mathbf{b}$ and from $\mathbf{b}$ to
    $\mathbf{a}$ in $W_k^{r,s}(n-1)$;
\item[(ii)] $W_k^{r,s}(n-1)$ is balanced.
\end{itemize}
\end{lemma}

\begin{proof}
\begin{itemize}
\item Suppose $\mathbf{a}=(a_0,a_1\dots,a_{n-1})$ and $\mathbf{b}=(b_0,b_1\dots,b_{n-1})$,
    where $a_i=b_i$ for every $i$ except $i=t$ for some $t$ ($0\leq t\leq n-1$).  Since $w$ is
    permutation-agnostic, if $\mathbf{a}'$ is a cyclic shift of $\mathbf{a}$ then
    $w(\mathbf{a})=w(\mathbf{a}')$.  Hence there is a directed path from $\mathbf{a}$ to
    $(a_t,a_{t+1},\dots,a_{n-1},a_0,a_1,\dots,a_{t-1})$ in $W_k^{r,s}(n-1)$. Clearly
    $(a_t,a_{t+1},\dots,a_{n-1},a_0,a_1,\dots,a_{t-1})$ is an incoming edge to the vertex
    $(a_{t+1},\dots,a_{n-1},a_0,a_1,\dots,a_{t-1})$ and
    $(a_{t+1},a_{t+2},\dots,a_{n-1},a_0,a_1,\dots,a_{t-1},b_t)=(b_{t+1},b_{t+2},\dots,b_{n-1},b_0,b_1,\dots,b_t)$
    is an outgoing edge from this vertex.  That is, we can extend the directed path so that it
    goes from $\mathbf{a}$ to a cyclic shift of $\mathbf{b}$. Using the same argument as
    previously, there is a directed path from
    $(b_{t+1},b_{t+2},\dots,b_{n-1},b_0,b_1,\dots,b_t)$ to $\mathbf{b}$ in $W_k^{r,s}(n-1)$.
    The same argument works in reverse, and the result follows.
\item $W_k^{r,r}(n-1)$ is balanced for any $r$, since if $(a_1,a_2,\dots,a_{n-1})$ is a vertex
    in $B_k(n-1)$ and $x\in \mathbb{Z}_k$, then $(x,a_1,a_2,\dots,a_{n-1})$ is an incoming edge
    in $W_k^{r,r}(n-1)$ if and only if $(a_1,a_2,\dots,a_{n-1},x)$ is an outgoing edge in
    $W_k^{r,r}(n-1)$ since $w$ is permutation-agnostic.  But the set of edges in
    $W_k^{r,s}(n-1)$ is simply the disjoint union of the sets of edges in $W_k^{t,t}(n-1)$ for
    every $t$ ($r\leq t\leq s$), and the result follows.
\end{itemize}
\end{proof}

\section{Partitioning using Hamming weight}

We start by considering the case where the weight of an $n$-tuple is simply the number of non-zero
entries, i.e.\ using the function $w_h$ as specified in Definition~\ref{definition:Hamming_weight}.
We have a result very similar to that of Sawada et al.~\cite{Sawada11} and Ruskey et
al.~\cite{Ruskey12} for the case $k=2$. We first need the following notation.

\begin{definition}[The Hamming weight subgraph]
Suppose $k\geq2$, $n\geq 1$ and $0\leq r\leq s\leq n$. We write $H_k^{r,s}(n-1)$ for the subgraph
of $B_k(n-1)$ consisting of the edges corresponding to $n$-tuples $\mathbf{a}$ satisfying $r\leq
w_h(\mathbf{a})\leq s$.
\end{definition}

\begin{remark}
$H_2^{r-1,r}(n-1)$ is what Ruskey et al.\ \cite{Ruskey12} refer to as $G(\mathbf{B}_{r-1}^r(n-1))$.
\end{remark}

The following two results are straightforward to establish.

\begin{lemma}  \label{lemma:H_numbers}
Suppose $k\geq2$, $n\geq 1$ and $0\leq r\leq s\leq n$.  Then the number of edges in
$H_k^{r,s}(n-1)$ is
\[ \sum_{i=r}^s  \binom{n}{i}(k-1)^i. \]
\end{lemma}

\begin{proof}
Observe that $H_k^{r,r}(n-1)$ contains
\[ \binom{n}{r}(k-1)^r \]
edges, since the number of such edges is simply the number of ways of choosing $r$ non-zero
elements from $n$ ($\binom{n}{r}$) multiplied by the number of choices for these non-zero elements
($(k-1)^r$). But the edges in $H_k^{r,s}(n-1)$ are simply all those in $H_k^{i,i}(n-1)$ for every
$i$ in the range $[r,s]$, and the result follows.
\end{proof}

\begin{lemma}  \label{lemma:H_successor rule}
Suppose $k\geq2$, $n\geq 2$, $0\leq r\leq n$ and $(a_1,a_2,\dots,a_{n-1})$ is a vertex in
$H_k^{r,r}(n-1)$.  If $(a_0,a_1,\dots,a_{n-1})$ is an incoming edge to this vertex then
$(a_1,a_2,\dots,a_{n})$ is an outgoing edge if and only if either $a_0=a_n=0$ or $a_0$ and $a_n$
are both non-zero.
\end{lemma}

\begin{proof}
This follows since all edges in $H_k^{r,r}(n-1)$ must correspond to $n$-tuples with precisely $r$
non-zero entries.
\end{proof}

Before giving the main result we need the following key lemma.

\begin{lemma}  \label{lemma:H_prelims}
Suppose $n\geq2$, $k\geq2$ and $0\leq r<n$.
\begin{itemize}
\item[(i)] If $\mathbf{a}$ is an edge (i.e.\ an $n$-tuple) in $H_k^{r,r+1}(n-1)$, then there is
    a directed path in $H_k^{r,r+1}(n-1)$ from $\mathbf{a}$ to an edge $\mathbf{a}'$ satisfying
    $w_h(\mathbf{a}')=r$.
\item[(ii)] If $\mathbf{b}$ is an edge (i.e.\ an $n$-tuple) in $H_k^{r,r+1}(n-1)$, then there
    is a directed path in $H_k^{r,r+1}(n-1)$ from an edge $\mathbf{b}'$ satisfying
    $w_h(\mathbf{b}')=r+1$ to $\mathbf{b}$.
\item[(iii)] If $\mathbf{a}=(a_0,a_1,\dots,a_{n-1})$ is an edge (i.e.\ an $n$-tuple) in
    $H_k^{r,r+1}(n-1)$ such that $w_h(\mathbf{a})=r$ and $a_i=0$ for some $i$, then there is a
    directed path in $H_k^{r,r+1}(n-1)$ from $\mathbf{a}$ to
\[ (a_0,a_1,\dots,a_{i-1},a_{i+1}, a_i,a_{i+2},\dots,a_{n-1}), \]
i.e.\ $\mathbf{a}$ modified so that $a_i$ and $a_{i+1}$ are interchanged.
\item[(iv)] If $\mathbf{b}=(b_0,b_1,\dots,b_{n-1})$ is an edge (i.e.\ an $n$-tuple) in
    $H_k^{r,r+1}(n-1)$ such that $w_h(\mathbf{b})=r+1$ and $b_i=0$ for some $i$, then there is
    a directed path in $H_k^{r,r+1}(n-1)$ from
\[ (b_0,b_1,\dots,b_{i-1},b_{i+1},b_{i},b_{i+2},\dots,b_{n-1}) \]
i.e.\ $\mathbf{b}$ modified so that $b_i$ and $b_{i+1}$ are interchanged, to $\mathbf{b}$.
\end{itemize}
\end{lemma}

\begin{proof}
\begin{itemize}
\item[(i)] If $\mathbf{a}=(a_0,a_1,\dots,a_{n-1})$ satisfies $w_h(\mathbf{a})=r+1$, then
    suppose $a_i\neq0$ for some $i$ (which must exist since $w_h(\mathbf{a})>0$).  Now clearly
    there is a directed path in $H_k^{r,r+1}(n-1)$ from $\mathbf{a}$ to
    $(a_i,a_{i+1},\dots,a_{n-1},a_0,a_1,\dots,a_{i-1})$, simply by taking successive cyclic
    shifts. Because $w_h(\mathbf{a})=r+1$ and $a_i\neq0$, the edge
    $(a_{i+1},\dots,a_{n-1},a_0,a_1,\dots,a_{i-1},0)$ is an edge in $H_k^{r,r+1}(n-1)$ with
    weight $r$, and so a directed path exists in $H_k^{r,r+1}(n-1)$ from $\mathbf{a}$ to an
    edge with weight $r$. Alternatively, if $w_h(\mathbf{a})=r$, then simply set
    $\mathbf{a}=\mathbf{a}'$. Either way the result holds.

\item[(ii)] If $\mathbf{b}=(b_0,b_1,\dots,b_{n-1})$ satisfies $w_h(\mathbf{b})=r<n$, then there
    is at least one value $b_i$ such that $b_i=0$.  As in the previous case, there is a
    directed path in $H_k^{r,r+1}(n-1)$ from
    $(b_i,b_{i+1},\dots,b_{n-1},b_0,b_1,\dots,b_{i-1})$ to $\mathbf{b}$. Because
    $w_h(\mathbf{b})=r$, the edge $(1,b_{i+1},\dots,b_{n-1},b_0,b_1,\dots,b_{i-1})$ is an edge
    in $H_k^{r,r+1}(n-1)$ with Hamming weight $r+1$, and the desired directed path exists.
    Alternatively, if $w_h(\mathbf{b})=r+1$, then simply set $\mathbf{b}=\mathbf{b}'$.  Either
    way the result holds.

\item[(iii)] There is clearly a directed path from $\mathbf{a}$ to
    $(a_i,a_{i+1},\dots,a_{n-1},a_0,a_1,\dots,a_{i-1})$ in $H_k^{r,r+1}(n-1)$, simply by taking
    successive cyclic shifts. Since $\mathbf{a}$ has weight $r$ and $a_i=0$, the $n$-tuple
    \[(a_{i+1},a_{i+2},\dots,a_{n-1},a_0,a_1,\dots,a_{i-1},a_{i+1})\]
    has weight $r$ or $r+1$ and hence is in $H_k^{r,r+1}(n-1)$.  Clearly
    \[(a_{i+2},a_{i+3},\dots,a_{n-1},a_0,a_1,\dots,a_{i-1},a_{i+1},a_i)\]
    has weight the same as $\mathbf{a}$ and hence is also in $H_k^{r,r+1}(n-1)$. Finally, there
    is a directed path from this latter $n$-tuple to
    \[ (a_0,a_1,\dots,a_{i-1},a_{i+1},a_{i},a_{i+2},\dots,a_{n-1}), \]
    via cyclic shifting and so the desired directed path exists.

\item[(iv)] There is a directed path from
    $(b_{i+2},b_{i+3},\dots,b_{n-1},b_0,b_1,\dots,b_{i+1})$ to $\mathbf{b}$ in
    $H_k^{r,r+1}(n-1)$, taking cyclic shifts.  Because $\mathbf{b}$ has weight $r+1$ and
    $b_i=0$, the $n$-tuple
    \[(b_i,b_{i+2},\dots,b_{n-1},b_0,b_1,\dots,b_i)\]
    has weight $r$ or $r+1$ and hence is in $H_k^{r,r+1}(n-1)$.  Also
    \[(b_{i+1},b_i,b_{i+2},\dots,b_{n-1},b_0,b_1,\dots,b_{i-1})\]
    has weight the same as $\mathbf{b}$ and so is likewise in $H_k^{r,r+1}(n-1)$.  This
    completes the proof.
\end{itemize}
\end{proof}

We can now give our main result.

\begin{theorem}  \label{theorem:Hamming_characterisation}
Suppose $k\geq2$ and $n\geq 2$. Then:
\begin{itemize}
\item[(i)] $H_k^{r,r}(n-1)$ is Eulerian if and only if $r\in\{0,1,n-1,n\}$;
\item[(ii)] $H_k^{r,r+1}(n-1)$ is Eulerian for every $r$ ($0\leq r<n$).
\end{itemize}
\end{theorem}

\begin{remark}  \label{remark:H_SWW}
Certain special cases of this result were already known. For the case $k=2$ this is precisely
\cite[Theorem 2.4]{Ruskey12}. Also, using the notation of
Section~\ref{subsection:Sawada-Williams-Wong}, the set of edges in $H_k^{r,n}(n-1)$ is in
$\mathcal{C}(n,k)$ for every $r$, $0\leq r\leq n$, since replacing entries of an $n$-tuple with
$k-1$ can never decrease the Hamming weight.  Hence, by Result~\ref{result:Sawada-Williams-Wong},
$H_k^{r,n}(n-1)$ is Eulerian.
\end{remark}

\begin{proof}
Because of Lemma~\ref{lemma:balanced}, establishing the Eulerian property just requires us to show
that the subgraphs are strongly connected.
\begin{itemize}
\item[(i)] The result is trivially true for $r=0$, since there is a single directed edge
    with Hamming weight 0, the loop to and from the vertex $0^{n-1}$.  If $r=n$, then the
    edges in $H_k^{n,n}(n-1)$ are all the $n$-tuples that do not contain a zero. But this
    subgraph is isomorphic to the full de Bruijn graph $B_{k-1}(n-1)$, and hence it is
    clearly strongly connected.

Now suppose $r=1$.  The edges of $H_k^{1,1}(n-1)$ correspond to all $n$-tuples containing a
single non-zero entry.  Suppose $\mathbf{a}=(a_0,a_1,\dots,a_{n-1})$ and
$\mathbf{b}=(b_0,b_1,\dots,b_{n-1})$ are two such $n$-tuples, where $a_u\neq0$ and $b_v\neq0$
and all the other $a_i$ and $b_i$ values are zero. Then there is clearly a directed path in
$H_k^{1,1}(n-1)$ from $\mathbf{a}$ to the $n$-tuple $(a_u,0,0,\dots,0)$ involving taking
successive cyclic shifts of $\mathbf{a}$, and similarly there is also a directed path from the
$n$-tuple $(0,0,\dots,0,b_v)$ to $\mathbf{b}$. Also, $(a_u,0,0,\dots,0)$ is an incoming edge to
the vertex $0^{n-1}$ and $(0,0,\dots,0,b_v)$ is an outgoing edge from the same vertex. Hence
there is a directed path from $\mathbf{a}$ to $\mathbf{b}$ and we have established strong
connectivity.

Finally, a somewhat similar argument applies when $r=n-1$.  Suppose
$\mathbf{a}=(a_0,a_1,\dots,a_{n-1})$ and $\mathbf{b}=(b_0,b_1,\dots,b_{n-1})$ are two
$n$-tuples corresponding to edges in $H_k^{n-1,n-1}(n-1)$, where $a_u=b_v=0$ and all the other
$a_i$ and $b_i$ values are non-zero.  Then there are directed paths in $H_k^{n-1,n-1}(n-1)$:
\begin{itemize}
\item from $\mathbf{a}$ to the $n$-tuple
    $(a_{u+1},a_{u+2},\dots,a_{n-1},a_0,a_1,\dots,a_{u-1},0)$, just by taking cyclic
    shifts;
\item then from $(a_{u+1},a_{u+2},\dots,a_{n-1},a_0,a_1,\dots,a_{u-1},0)$ to
    \\$(0,b_{v+1},b_{v+2},\dots,b_{n-1},b_0,b_1,\dots,b_{v-1})$ (from repeated application of
    Lemma~\ref{lemma:H_successor rule}); and
\item finally from $(0,b_{v+1},b_{v+2},\dots,b_{n-1},b_0,b_1,\dots,b_{v-1})$ to
    $\mathbf{b}$, again taking cyclic shifts;
\end{itemize}
yielding the required directed path from $\mathbf{a}$ to $\mathbf{b}$.

Finally suppose $1<r<n-1$, and thus $n\geq4$ --- we need to show that $H_k^{r,r}(n-1)$ is not
strongly connected.  Regardless of $r$ there will exist two edges in $H_k^{r,r}(n-1)$
corresponding to the $n$-tuples $\mathbf{a}=(1^r,0^{n-r})$ and
$\mathbf{b}=(1,0,1^{r-1},0^{n-r-1})$, and all the constituent strings are non-empty since
$n\geq 4$ and $1<r<n-1$.  We claim that there is no path in $H_k^{r,r}(n-1)$ linking
$\mathbf{a}$ to $\mathbf{b}$.

This holds because, from Lemma~\ref{lemma:H_successor rule}, any edge on a path in
$H_k^{r,r}(n-1)$ starting from $\mathbf{a}$ must contain $r$ consecutive non-zero entries
(working cyclically). Similarly, any edges on a path starting from $\mathbf{b}$ must contain a
zero followed by $r-1$ consecutive non-zero entries and then another zero (again working
cyclically). That is $\mathbf{b}$ cannot be on a path starting at $\mathbf{a}$, and vice versa.

\item[(ii)] Suppose $\mathbf{a}$ and $\mathbf{b}$ are edges in $H_k^{r,r+1}(n-1)$.  Then, by
    Lemma~\ref{lemma:H_prelims}(i) and (ii), there are directed paths in $H_k^{r,r+1}(n-1)$
    from $\mathbf{a}$ to $\mathbf{a}'$ and from $\mathbf{b}'$ to $\mathbf{b}$, where
    $\mathbf{a}'$ has weight $r$ and $\mathbf{b}'$ has weight $r+1$.  So to establish the
    result we simply need to show there is a path in $H_k^{r,r+1}(n-1)$ from any $n$-tuple with
    weight $r$ to any $n$-tuple with weight $r+1$.

Suppose $\mathbf{a}$ has weight $r$ and $\mathbf{b}$ has weight $r+1$, where $0\leq r<n$. Then
we claim there are directed paths in $H_k^{r,r+1}(n-1)$ from $\mathbf{a}$ to $(1^r,0^{n-r})$,
and from $(1^{r+1},0^{n-r-1})$ to $\mathbf{b}$. The first assertion follows first from repeated
application of Lemma~\ref{lemma:H_prelims}(iii) to $\mathbf{a}$ to give a path from
$\mathbf{a}$ to an edge whose first $r$ entries are non-zero and the last $n-r$ entries are
zero.  A path from this vertex to $(1^r,0^{n-r})$ exists from Lemma~\ref{lemma:H_successor
rule}.

An exactly analogous argument using repeated applications of
Lemmas~\ref{lemma:H_prelims}(iv) and \ref{lemma:H_successor rule} shows that there is a
directed path from $(1^{r+1},0^{n-r-1})$ to $\mathbf{b}$.  That is, to complete the proof
it remains to show that there is a directed path in $H_k^{r,r+1}(n-1)$ from
$(1^{r},0^{n-r})$ to $(1^{r+1},0^{n-r-1})$; but this is immediate since they differ in only
one digit and hence the result follows from Lemma~\ref{lemma:one_change}.

\end{itemize}
\end{proof}

\begin{remark}  \label{remark:zero_patterns}
A simpler somewhat less formal approach to proving part (i) of the above theorem is to observe that
the $n$-tuples corresponding to connected edges in $H_k^{r,r}(n-1)$ must all have the same
`pattern' of occurrences of zeros, up to cyclic shift.  This follows from
Lemma~\ref{lemma:H_successor rule}.  Thus, since $H_k^{r,r}(n-1)$ contains $n$-tuples with
different `patterns' of zeros unless $r\in\{0,1,n-1,n\}$, the result follows.
\end{remark}

Theorem~\ref{theorem:Hamming_characterisation} has the following immediate corollary.

\begin{corollary}  \label{corollary:H_characterisation}
Suppose $k\geq2$, $n\geq 2$ and $0\leq r<s\leq n$. Then $H_k^{r,s}(n-1)$ is Eulerian.
\end{corollary}

\begin{proof}
The fact that $H_k^{r,s}(n-1)$ is balanced is immediate from Lemma~\ref{lemma:balanced}. We know
that the set of edges in $H_k^{r,s}(n-1)$ is the union of the sets of edges in $H_k^{t,t+1}(n-1)$
for every $t$ satisfying $r\leq t<s$, and $H_k^{t,t+1}(n-1)$ is strongly connected from
Theorem~\ref{theorem:Hamming_characterisation}(ii).  So suppose $\mathbf{a}_t$ and $\mathbf{a}_u$
are edges in $H_k^{r,s}(n-1)$ of weights $t$ and $u$ respectively.  Without loss of generality
suppose $t\leq u$.  If $u=t$ or $u=t+1$ then both $\mathbf{a}_t$ and $\mathbf{a}_u$ are in
$H_k^{t,t+1}(n-1)$, and hence there is a directed path from $\mathbf{a}_t$ to $\mathbf{a}_u$. So
suppose $u>t+1$.  Then let $\mathbf{a}_{t+1},\mathbf{a}_{t+2},\dots,\mathbf{a}_{u-1}$ be arbitrary
edges of weights $t+1,t+2,\dots,u-1$ respectively.  Then since $H_k^{i,i+1}(n-1)$ is strongly
connected for every $i$, there exist directed paths in $H_k^{r,s}(n-1)$ from $\mathbf{a}_i$ to
$\mathbf{a}_{i+1}$ for every $i$ ($t\leq i<u$), and hence from $\mathbf{a}_t$ to $\mathbf{a}_u$. It
follows that $H_k^{r,s}(n-1)$ is strongly connected.
\end{proof}

\begin{example}  \label{example:Hamming}
We give examples of cut-down de Bruijn sequences arising from Eulerian circuits in subgraphs
$H_k^{r,s}(n-1)$; in each case one period of the sequence is shown.
\begin{itemize}
\item[(i)] Let $k=n=3$ and $r=s=2=n-1$. By Theorem~\ref{theorem:Hamming_characterisation}(i),
    $H_3^{2,2}(2)$ is Eulerian, and it has $\binom{3}{2}2^2=12$ edges. The sequence
    \[ (101202102201) \]
    of period 12 contains every ternary 3-tuple of Hamming weight 2 exactly once in a period.
\item[(ii)] Let $k=n=3$, $r=1$ and $s=2$. By
    Theorem~\ref{theorem:Hamming_characterisation}(ii), $H_3^{1,2}(2)$ is Eulerian, and it has
    $\binom{3}{1}2+\binom{3}{2}2^2=18$ edges. The sequence
    \[ (100201101202102200) \]
    of period 18 contains every ternary 3-tuple of Hamming weight 1 or 2 exactly once in a
period.
\item[(iii)] Let $k=3$, $n=4$, $r=1$ and $s=2$. Then $H_3^{1,2}(3)$ has
    $\binom{4}{1}2+\binom{4}{2}2^2=32$ edges, and the sequence
    \[ (10002010102020011001200210022000) \]
    of period 32 contains every ternary 4-tuple of Hamming weight 1 or 2 exactly once in a
period.
\item[(iv)] Theorem~\ref{theorem:Hamming_characterisation}(i) shows that $H_3^{2,2}(3)$ is not
    Eulerian. Indeed, as observed in Remark~\ref{remark:zero_patterns}, the 24 edges of
    $H_3^{2,2}(3)$ fall into two connected components according to the cyclic pattern of their
    zeros: the 16 4-tuples whose two non-zero entries are cyclically adjacent (such as
    $(1,2,0,0)$), and the 8 4-tuples whose two non-zero entries are not (such as $(1,0,2,0)$).
\end{itemize}
\end{example}

\section{Partitioning using digit sum}

The second `measure' for $n$-tuples that we consider is the digit sum, as in
Definition~\ref{definition:digit_sum}.  For this case we need the following.

\begin{definition}[The digit sum subgraph]
Suppose $k\geq2$, $n\geq 2$ and $0\leq r\leq s\leq n(k-1)$. We write $S_k^{r,s}(n-1)$ for the
subgraph of $B_k(n-1)$ consisting of the edges corresponding to $n$-tuples $\mathbf{a}$ satisfying
$r\leq w_s(\mathbf{a})\leq s$.
\end{definition}

A special case of the problem we address here has been previously considered for this
function.

\begin{result}[\cite{Gabric20}]  \label{result:old_S_lemma}
Suppose $k\geq2$ and $n\geq 2$. Then $S_k^{0,s}(n-1)$ and $S_k^{s,n(k-1)}(n-1)$ are Eulerian,
for every $s$ ($0\leq s\leq n(k-1)$).
\end{result}

\begin{remark}  \label{remark:S_SWW}
It is interesting to observe that Result~\ref{result:old_S_lemma} is briefly stated in earlier work
by Sawada et al.~\cite[Section 3.3.1]{Sawada16}. Using the notation of
Section~\ref{subsection:Sawada-Williams-Wong}, the set of edges in $S_k^{s,n(k-1)}(n-1)$ is in
$\mathcal{C}(n,k)$ for every $s$, $0\leq s\leq n(k-1)$, since replacing entries of an $n$-tuple
with $k-1$ can never decrease the digit sum. Hence, by Result~\ref{result:Sawada-Williams-Wong},
$S_k^{s,n(k-1)}(n-1)$ is Eulerian. The fact that $S_k^{0,s}(n-1)$ is Eulerian follows by mapping
any symbol $i$ to $k-1-i$, which is an isomorphism of the de Bruijn graph.
\end{remark}

In fact a much more general result can be established.  We first observe the following.

\begin{lemma} \label{lemma:S_numbers}
Suppose $k\geq2$, $n\geq 1$ and $0\leq s\leq n(k-1)$.  Then the number of edges in $S_k^{s,s}(n-1)$
is
\[
N_k(n,s) = \sum_{j=0}^{\lfloor s/k\rfloor}(-1)^j\binom{n}{j}\binom{s-jk+n-1}{n-1}.
\]
Equivalently, $N_k(n,s)$ is the coefficient of $x^s$ in
$(1+x+\cdots+x^{k-1})^n=\left(\frac{1-x^k}{1-x}\right)^n$.
\end{lemma}

\begin{proof}
Let $U$ be the set of $n$-tuples of non-negative integers (not necessarily from $[0,k-1]$) whose
digit sum equals $s$. A standard `stars-and-bars' argument gives $|U|=\binom{s+n-1}{n-1}$. For
$i\in\{1,2,\dots,n\}$, let $A_i=\{(a_1,\ldots,a_n)\in U : a_i\ge k\}$. The tuples we wish to count
are exactly those in $U-\bigcup_{i=1}^n A_i$.

Fix $J\subseteq\{1,2,\dots,n\}$ with $|J|=j$. The mapping that subtracts $k$ from each coordinate
indexed by $J$ is a bijection from $\bigcap_{i\in J}A_i$ to the set of $n$-tuples of non-negative
integers with sum $s-jk$. Hence
\[
\Bigl|\bigcap_{i\in J}A_i\Bigr|=\binom{s-jk+n-1}{n-1}
\]
when $jk\leq s$, and the intersection is empty otherwise. By inclusion-exclusion,
\[
N_k(n,s)=\sum_{J\subseteq\{1,\ldots,n\}}(-1)^{|J|}
\Bigl|\bigcap_{i\in J}A_i\Bigr|
=\sum_{j=0}^{\lfloor s/k\rfloor}(-1)^j\binom{n}{j}\binom{s-jk+n-1}{n-1}.
\]

For the generating-function form, note that $N_k(n,s)$ is the number of ways of choosing one term
$x^{a_i}$ from each of the $n$ factors of $(1+x+\cdots+x^{k-1})^n$ so that the exponents sum to
$s$. That is, $N_k(n,s)=[x^s](1+x+\cdots+x^{k-1})^n$. Expanding $(1-x^k)^n(1-x)^{-n}$ by the
binomial theorem gives the same sum.
\end{proof}

\begin{lemma}  \label{lemma:S_successor rule}
Suppose $k\geq2$, $n\geq 2$, $0\leq r\leq n(k-1)$ and $(a_1,a_2,\dots,a_{n-1})$ is a vertex in
$S_k^{r,r}(n-1)$. If $(a_0,a_1,\dots,a_{n-1})$ is an incoming edge to this vertex then
$(a_1,a_2,\dots,a_{n})$ is an outgoing edge if and only if $a_0=a_n$.
\end{lemma}

\begin{proof}
This is immediate since all edges in $S_k^{r,r}(n-1)$ correspond to $n$-tuples with digit sum $r$.
\end{proof}

Since $w_s(\mathbf{a})=w_h(\mathbf{a})$ if $k=2$, Theorem~\ref{theorem:Hamming_characterisation}
applies in this case.  Thus we restrict our attention here to the case $k>2$.  Before giving the
main result we need the following key lemma.

\begin{lemma}  \label{lemma:S_prelims}
Suppose $n\geq2$, $k>2$ and $0\leq r<n(k-1)$.
\begin{itemize}
\item[(i)] If $\mathbf{a}$ is an edge (i.e.\ an $n$-tuple) in $S_k^{r,r+1}(n-1)$, then there is
    a directed path in $S_k^{r,r+1}(n-1)$ from $\mathbf{a}$ to an edge $\mathbf{a}'$ satisfying
    $w_s(\mathbf{a}')=r$.
\item[(ii)] If $\mathbf{b}$ is an edge (i.e.\ an $n$-tuple) in $S_k^{r,r+1}(n-1)$, then there
    is a directed path in $S_k^{r,r+1}(n-1)$ from an edge $\mathbf{b}'$ satisfying
    $w_s(\mathbf{b}')=r+1$ to $\mathbf{b}$.
\item[(iii)] If $\mathbf{a}=(a_0,a_1,\dots,a_{n-1})$ is an edge (i.e.\ an $n$-tuple) in
    $S_k^{r,r+1}(n-1)$ such that $w_s(\mathbf{a})=r$, $a_i<k-1$ and $a_{i+1}>0$, for some $i$
    ($0\leq i<n-1$), then there is a directed path in $S_k^{r,r+1}(n-1)$ from $\mathbf{a}$ to
\[ (a_0,a_1,\dots,a_{i-1},a_i+1, a_{i+1}-1,a_{i+2},\dots,a_{n-1}), \]
i.e.\ $\mathbf{a}$ modified so that $a_i$ is increased by 1 and $a_{i+1}$ decreased by 1.
\item[(iv)] If $\mathbf{b}=(b_0,b_1,\dots,b_{n-1})$ is an edge (i.e.\ an $n$-tuple) in
    $S_k^{r,r+1}(n-1)$ such that $w_s(\mathbf{b})=r+1$, $b_i<k-1$ and $b_{i+1}>0$, for some $i$
    ($0\leq i<n-1$), then there is a directed path in $S_k^{r,r+1}(n-1)$ from
\[ (b_0,b_1,\dots,b_{i-1},b_i+1,b_{i+1}-1,b_{i+2},\dots,b_{n-1}) \]
i.e.\ $\mathbf{b}$ modified so that $b_i$ is increased by 1 and $b_{i+1}$ decreased by 1, to
$\mathbf{b}$.
\end{itemize}
\end{lemma}

\begin{proof}
\begin{itemize}
\item[(i)] If $\mathbf{a}=(a_0,a_1,\dots,a_{n-1})$ satisfies $w_s(\mathbf{a})=r+1$, then
    suppose $a_i\neq0$ for some $i$ (which must exist since $w_s(\mathbf{a})>0$).  Now using
    cyclic shifting there is a directed path in $S_k^{r,r+1}(n-1)$ from $\mathbf{a}$ to
    $(a_i,a_{i+1},\dots,a_{n-1},a_0,a_1,\dots,a_{i-1})$. Because $w_s(\mathbf{a})=r+1$, the
    edge $(a_{i+1},\dots,a_{n-1},a_0,a_1,\dots,a_{i-1},a_{i}-1)$ is an edge in
    $S_k^{r,r+1}(n-1)$ with digit sum $r$, and so a directed path exists in $S_k^{r,r+1}(n-1)$
    from $\mathbf{a}$ to an edge with digit sum $r$. Alternatively, if $w_s(\mathbf{a})=r$,
    then simply set $\mathbf{a}=\mathbf{a}'$. Either way the result holds.
\item[(ii)] If $\mathbf{b}=(b_0,b_1,\dots,b_{n-1})$ satisfies $w_s(\mathbf{b})=r<n(k-1)$, then
    there is at least one value $b_i$ such that $b_i<k-1$.  As in the previous case, using
    cyclic shifting there is a directed path in $S_k^{r,r+1}(n-1)$ from
    $(b_i,b_{i+1},\dots,b_{n-1},b_0,b_1,\dots,b_{i-1})$ to $\mathbf{b}$. Because
    $w_s(\mathbf{b})=r$, the edge $(b_i+1,b_{i+1},\dots,b_{n-1},b_0,b_1,\dots,b_{i-1})$ is an
    edge in $S_k^{r,r+1}(n-1)$ with digit sum $r+1$, and the desired directed path exists.
    Alternatively, if $w_s(\mathbf{b})=r+1$, then simply set $\mathbf{b}=\mathbf{b}'$.  Either
    way the result holds.
\item[(iii)] There is clearly a directed path from $\mathbf{a}$ to
    $(a_i,a_{i+1},\dots,a_{n-1},a_0,a_1,\dots,a_{i-1})$ in $S_k^{r,r+1}(n-1)$ simply from
    cyclic shifting. Since $\mathbf{a}$ has digit sum $r$, the $n$-tuple
    \[(a_{i+1},a_{i+2},\dots,a_{n-1},a_0,a_1,\dots,a_{i-1},a_i+1)\]
    has digit sum $r+1$ and hence is in $S_k^{r,r+1}(n-1)$.  It follows that
    \[(a_{i+2},a_{i+3},\dots,a_{n-1},a_0,a_1,\dots,a_{i-1},a_i+1,a_{i+1}-1)\]
    has digit sum $r$ and hence is also in $S_k^{r,r+1}(n-1)$. Finally, there is a directed
    path from this latter $n$-tuple to
    \[ (a_0,a_1,\dots,a_{i-1},a_i+1, a_{i+1}-1,a_{i+2},\dots,a_{n-1}), \]
    via cyclic shifting and so the desired directed path exists.
\item[(iv)] There is a directed path from
    $(b_{i+2},b_{i+3},\dots,b_{n-1},b_0,b_1,\dots,b_{i+1})$ to $\mathbf{b}$ in
    $S_k^{r,r+1}(n-1)$ using cyclic shifting.  Because  $\mathbf{b}$ has digit sum $r+1$, the
    $n$-tuple
\[(b_{i+1}-1,b_{i+2},\dots,b_{n-1},b_0,b_1,\dots,b_i)\]
has digit sum $r$ and hence is in $S_k^{r,r+1}(n-1)$.  Hence
\[(b_i+1,b_{i+1}-1,b_{i+2},\dots,b_{n-1},b_0,b_1,\dots,b_{i-1})\]
has digit sum $r+1$ and so is likewise in $S_k^{r,r+1}(n-1)$. Finally, there is a directed path
    from
    \[ (b_0,b_1,\dots,b_{i-1},b_i+1, b_{i+1}-1,b_{i+2},\dots,b_{n-1}), \]
    to this latter $n$-tuple via cyclic shifting and so the desired directed path exists.  This
    completes the proof.
\end{itemize}
\end{proof}

We can now give the following result.

\begin{theorem}  \label{theorem:S_characteristic}
Suppose $n\geq2$ and $k>2$. Then:
\begin{itemize}
\item[(i)] $S_k^{r,r}(n-1)$ is Eulerian if and only if $r\in\{0,1,(k-1)n-1,(k-1)n\}$;
\item[(ii)] $S_k^{r,r+1}(n-1)$ is Eulerian for every $r$ ($0\leq r<(k-1)n$).
\end{itemize}
\end{theorem}

\begin{proof}
\begin{itemize}
\item[(i)] If $r=0$ or $r=(k-1)n$ then the only $n$-tuple with digit sum $r$ is $0^n$ or
    $(k-1)^n$, respectively.  The result follows.  Somewhat similarly, the only $n$-tuple with
    digit sum $r=1$ consists of $n-1$ zeros and a single 1 symbol.  Clearly this set of $n$
    $n$-tuples can be arranged in a single cycle --- a necklace.  A precisely analogous
    argument holds for $r=(k-1)n-1$.

    Now suppose $1<r<(k-1)n-1$.  By Lemma~\ref{lemma:S_successor rule}, the only cycles in
$S_k^{r,r}(n-1)$ consist of necklaces, and thus have period dividing $n$. However, if
$1<r<(k-1)n-1$ then $S_k^{r,r}(n-1)$ will clearly contain more than $n$ edges ($n$-tuples), and
hence it contains more than one necklace.  However, two necklaces in $S_k^{r,r}(n-1)$ cannot
share a vertex since there is only at most one edge of weight $r$ outgoing from (and incoming
to) any vertex. The desired result follows.
\item[(ii)] Just as previously, because of Lemma~\ref{lemma:balanced}, establishing the
    Eulerian property just requires us to show that the subgraphs are strongly connected.
    Suppose $\mathbf{a}$ and $\mathbf{b}$ are edges in $S_k^{r,r+1}(n-1)$.  Then, by
    Lemma~\ref{lemma:S_prelims}(i) and (ii), there are directed paths in $S_k^{r,r+1}(n-1)$
    from $\mathbf{a}$ to $\mathbf{a}'$ and from $\mathbf{b}'$ to $\mathbf{b}$, where
    $\mathbf{a}'$ has digit sum $r$ and $\mathbf{b}'$ has digit sum $r+1$.  So to establish the
    result we simply need to show there is a path in $S_k^{r,r+1}(n-1)$ from any $n$-tuple with
    digit sum $r$ to any $n$-tuple with digit sum $r+1$.

Suppose $\mathbf{a}$ has digit sum $r$ and $\mathbf{b}$ has digit sum $r+1$, where
$r=s(k-1)+t$, $s\geq 0$, and $0\leq t<k-1$. Then we claim there are directed paths in
$S_k^{r,r+1}(n-1)$ from $\mathbf{a}$ to $((k-1)^s,t,0^{n-s-1})$, and from
$((k-1)^s,t+1,0^{n-s-1})$ to $\mathbf{b}$.  The first assertion follows from repeated
application of Lemma~\ref{lemma:S_prelims}(iii) to $\mathbf{a}$ until an $n$-tuple
$(c_0,c_1,\dots,c_{n-1})$ of digit sum $r$ is reached for which there is no $i$ such that
$c_i<k-1$ and $c_{i+1}>0$; this process terminates since each step increases the tuple
lexicographically. It follows that $\mathbf{c}=((k-1)^s,t,0^{n-s-1})$, and the first claim is
established.  An exactly analogous argument using repeated applications of
Lemma~\ref{lemma:S_prelims}(iv) shows that there is a directed path from
$((k-1)^s,t+1,0^{n-s-1})$ to $\mathbf{b}$.  That is, to complete the proof it remains to show
that there is a directed path in $S_k^{r,r+1}(n-1)$ from $((k-1)^s,t,0^{n-s-1})$ to
$((k-1)^s,t+1,0^{n-s-1})$; but this is immediate since they differ in only one digit, and hence
the result follows from Lemma~\ref{lemma:one_change}.
\end{itemize}
\end{proof}

Clearly Result~\ref{result:old_S_lemma} follows from Theorem~\ref{theorem:S_characteristic} as
a special case.

Analogously to the Hamming weight case, we have the following simple corollary, whose proof is
immediate using the same arguments as for Corollary~\ref{corollary:H_characterisation}.

\begin{corollary} \label{corollary:S_characteristic}
Suppose $n\geq2$, $k>2$ and $0\leq r< s\leq (k-1)n$. Then $S_k^{r,s}(n-1)$ is Eulerian.
\end{corollary}

\begin{example}  \label{example:digit_sum}
We give examples of cut-down de Bruijn sequences arising from Eulerian circuits in subgraphs
$S_k^{r,s}(n-1)$; in each case one period of the sequence is shown.
\begin{itemize}
\item[(i)] Let $k=n=3$. By Theorem~\ref{theorem:S_characteristic}(i), $S_3^{2,2}(2)$ is not
    Eulerian. Indeed, its six edges form the two necklaces $[0,0,2]$ and $[0,1,1]$, each of
    which is a directed cycle of length 3 in $B_3(2)$, and these two cycles are disjoint.
\item[(ii)] By Theorem~\ref{theorem:S_characteristic}(ii), $S_3^{2,3}(2)$ is Eulerian. It has
    $6+7=13$ edges, and the sequence
    \[ (2011102101200) \]
    of period 13 contains every ternary 3-tuple with digit sum 2 or 3 exactly once in a period.
\item[(iii)] Let $k=4$ and $n=3$. Then $S_4^{3,4}(2)$ has $10+12=22$ edges, and the sequence
    \[ (3012021112103102201300) \]
    of period 22 contains every 4-ary 3-tuple with digit sum 3 or 4 exactly once in a period.
\item[(iv)] Let $k=3$ and $n=4$. Then $S_3^{3,4}(3)$ has $16+19=35$ edges, and the sequence
    \[ (20021011110202011201021101210022001) \]
    of period 35 contains every ternary 4-tuple with digit sum 3 or 4 exactly once in a period.
\end{itemize}
\end{example}

\section{Partitioning using pseudoweight}

The third `measure' for $n$-tuples that we consider is the pseudoweight, as in
Definition~\ref{definition:pseudoweight}.  If $k=2$ then $w_p(0)=w_p(1)=1$ and the pseudoweight
function is trivial and uninteresting. Thus, in the remainder of this discussion we restrict our
attention to $k>2$. We need the following.

\begin{definition}[The pseudoweight subgraph]
Suppose $k>2$, $n\geq 2$ and $n\leq r\leq s\leq (k-1)n$ (where $r\in\mathbb{Z}$ if $k$ is even and
$r=t/2$ for some $t\in\mathbb{Z}$ if $k$ is odd). We write $P_k^{r,s}(n-1)$ for the subgraph of
$B_k(n-1)$ consisting of the edges corresponding to $n$-tuples $\mathbf{a}$ satisfying $r\leq
w_p(\mathbf{a})\leq s$.
\end{definition}

Exactly as for digit sum, special cases of the problem we address here have been previously
considered for this function. The discussion following \cite[Theorem 3.8]{Mitchell25a}
establishes the following result.

\begin{result}[\cite{Mitchell25a}]  \label{result:old_P_lemma}
Suppose $k>2$ and $n\geq 2$. Then $P_k^{n,s}(n-1)$ is Eulerian, where
\[ s = \begin{cases}
\frac{nk-1}{2}\text{~~if~~$k$ is odd}\\
\frac{nk-2}{2}\text{~~if~~$k$ is even}
\end{cases}\]
\end{result}

\begin{remark}
Since $\frac{nk}{2}$ is the average value of the function $w_p$ when calculated over all
$n$-tuples, this simply says that the subgraph of $B_k(n-1)$ consisting of all $n$-tuples with
pseudoweight less than the average value is Eulerian.
\end{remark}

\begin{remark}  \label{remark:P_SWW}
In fact the earlier work of Sawada et al.~\cite{Sawada16} can be used to give a stronger result
than Result~\ref{result:old_P_lemma}. Using the notation of
Section~\ref{subsection:Sawada-Williams-Wong}, the set of edges in $P_k^{s,n(k-1)}(n-1)$ is in
$\mathcal{C}(n,k)$ for every $s$, $n\leq s\leq n(k-1)$, since replacing entries of an $n$-tuple
with $k-1$ can never decrease the pseudoweight. Hence, by Result~\ref{result:Sawada-Williams-Wong},
$P_k^{s,n(k-1)}(n-1)$ is Eulerian. Mapping every symbol $i$ to $k-i$ (modulo $k$) gives an
isomorphism of the de Bruijn graph which maps $P_k^{s,n(k-1)}(n-1)$ to $P_k^{n,nk-s}(n-1)$, since
$w_p(k-i)=k-w_p(i)$ for every $i$. Hence $P_k^{n,s}(n-1)$ is also Eulerian for every $s$ satisfying
$n\leq s\leq n(k-1)$.
\end{remark}

First, we have the following.  Suppose $k>2$, $n\geq 1$ and $0\leq i\leq n(k-1)$ (where the range
includes values of the form $u/2$ where $u$ is an integer). Following~\cite{Mitchell25a}, let the
number of edges in $P_k^{i,i}(n-1)$ be denoted by $r_{k,n,i}$.  Then we have the following result.

\begin{result}[Lemma 3.10 of \cite{Mitchell25a}]  \label{result:P_numbers}
Suppose $k>2$ and $n\geq 1$.
\begin{enumerate}
\item[(i)] If $k$ is even, $r_{k,1,i}=1$ for $1\leq i\leq k-1$, $i\neq k/2$, and
    $r_{k,1,k/2}=2$.
\item[(ii)] If $k$ is odd, $r_{k,1,i}=1$ for $1\leq i\leq k-1$, and $r_{k,1,k/2}=1$.
\item[(iii)] If $k$ is even, $\sum_{i=n}^{n(k-1)}r_{k,n,i} = k^n$.
\item[(iv)] If $k$ is odd, $\sum_{j=0}^{2n(k-2)}r_{k,n,n+j/2} = k^n$.
\item[(v)] If $n>1$, $r_{k,n,s} = \sum_{i=1}^{k-1}r_{k,n-1,s-i}+r_{k,n-1,s-k/2}$.
\end{enumerate}
\end{result}

Result~\ref{result:P_numbers}(v) gives a simple recursion that, combined with (i) and (ii), enable
the number of edges in $P_k^{i,i}(n-1)$ to be readily computed for any $i$.

\begin{lemma}  \label{lemma:P_successor rule}
Suppose $k>2$, $n\geq 2$, $n\leq r\leq n(k-1)$ and $(a_1,a_2,\dots,a_{n-1})$ is a vertex in
$P_k^{r,r}(n-1)$.  If $(a_0,a_1,\dots,a_{n-1})$ is an incoming edge to this vertex then
$(a_1,a_2,\dots,a_{n})$ is an outgoing edge if and only if:
\[ a_n=\begin{cases}
0 \text{~or~} k/2 \text{~~if $k$ is even and $a_0=0$ or $a_0=k/2$}\\
a_0               \text{~~otherwise}
\end{cases} \]
\end{lemma}

\begin{proof}
Since both $(a_0,a_1,\dots,a_{n-1})$ and $(a_1,a_2,\dots,a_{n})$ have the same pseudoweight, it
follows that $f(a_0)=f(a_n)$, where $f$ is as defined in Definition~\ref{definition:pseudoweight}.
Hence, unless $k$ is even and $f(a_0)=f(a_n)=k/2$, this automatically implies $a_0=a_n$.
Conversely, if $f(a_0)=f(a_n)$ then $(a_0,a_1,\dots,a_{n-1})$ and $(a_1,a_2,\dots,a_{n})$ have the
same pseudoweight, and the result follows.
\end{proof}

The analysis for this function is somewhat more complex.  However, we do have a result similar to
Theorems~\ref{theorem:Hamming_characterisation}(i) and \ref{theorem:S_characteristic}.

\begin{theorem}  \label{theorem:P_characteristic_a}
Suppose $n\geq2$, $k>2$, $n\leq r\leq n(k-1)$ (where $r\in\mathbb{Z}$ if $k$ is even and $r=t/2$
for some $t\in\mathbb{Z}$ if $k$ is odd), and $P_k^{r,r}(n-1)$ is non-empty. Then $P_k^{r,r}(n-1)$
is Eulerian if and only if
\begin{itemize}
\item[(i)] $r\in\{n,(k-1)n\}$,
\item[(ii)] $r\in\{n+1,(k-1)n-1\}$ and $k\neq3$,
\item[(iii)]  $r\in\{n-1+k/2,(k-1)(n-1)+k/2\}$ and $k$ odd,
\item[(iv)]  $r\in\{i+k/2 : 1\le i \le k-1 \}$, $n=2$ and $k$ odd.
\end{itemize}
\end{theorem}

\begin{remark}
Note that when $n=2$, the cases $i=1$ and $i=k-1$ of (iv) coincide with (iii).  Also, the
requirement that $P_k^{r,r}(n-1)$ is non-empty only excludes the case where $r$ is not an integer
and $n<r<n-1+k/2$ or $(k-1)(n-1)+k/2<r<(k-1)n$, which can only happen when $k\geq5$ is odd.
\end{remark}

\begin{proof}
Because of Lemma~\ref{lemma:balanced}, to show the Eulerian property holds it suffices to show that
the relevant subgraph is strongly connected.
\begin{itemize}
\item[(i)] $P_k^{n,n}(n-1)$ consists of a single edge, namely $1^n$ --- a loop from the vertex
    $1^{n-1}$ to itself; this is clearly Eulerian.  Similarly, the subgraph
    $P_k^{(k-1)n,(k-1)n}(n-1)$ consists of the single edge $(k-1)^n$ and thus is also Eulerian.
\item[(ii)] $P_4^{n+1,n+1}(n-1)$ contains all the $n$-tuples in the two necklaces $[1^{n-1},2]$
    and $[1^{n-1},0]$.  These two necklaces share a vertex, namely $1^{n-1}$, and hence the
    subgraph is strongly connected. If $k>4$, then $P_k^{n+1,n+1}(n-1)$ contains only a single
    necklace, namely $[1^{n-1},2]$, and thus the subgraph is strongly connected.  Analogous
    arguments hold for $P_4^{3n-1,3n-1}(n-1)$, which contains the necklaces $[3^{n-1},2]$ and
    $[3^{n-1},0]$, and $P_k^{n(k-1)-1,n(k-1)-1}(n-1)$, which contains the single necklace
    $[(k-1)^{n-1},k-2]$.
\item[(iii)] If $k$ is odd then $n-1+k/2$ is not an integer, and hence the only necklace in
    $P_k^{n-1+k/2,n-1+k/2}(n-1)$ is $[1^{n-1},0]$, and so $P_k^{n-1+k/2,n-1+k/2}(n-1)$ is
    Eulerian.  Similarly, $(k-1)(n-1)+k/2$ is not an integer, and hence the only necklace in
    $P_k^{(k-1)(n-1)+k/2,(k-1)(n-1)+k/2}(n-1)$ is $[(k-1)^{n-1},0]$ and so
    $P_k^{(k-1)(n-1)+k/2,(k-1)(n-1)+k/2}(n-1)$ is Eulerian.
\item[(iv)] If $n=2$ and $k$ is odd then when $r=i+k/2$ with $1\le i \le k-1$
    $P_k^{i+k/2,i+k/2}(n-1)$ only contains the edges $(0,i),(i,0)$, a necklace of period $2$,
    and so $P_k^{i+k/2,i+k/2}(n-1)$ is Eulerian.  This holds since if a tuple has a non-integer
    weight then it must contain an odd number of zeros, i.e.\ a single zero when $n=2$.
\end{itemize}
It remains to show that in all other cases $P_k^{r,r}(n-1)$ is not Eulerian. First consider the
special case $k=3$ and $r\in\{n+1,(k-1)n-1\}$. Observe that $P_3^{n+1,n+1}$ contains two necklaces:
$[1^{n-1},2]$ and $[1^{n-2},0,0]$. These clearly do not share a vertex, and hence $P_3^{n+1,n+1}$
is not Eulerian.  Similarly, $P_3^{2n-1,2n-1}(n-1)$ contains the two necklaces $[2^{n-1},1]$ and
$[2^{n-2},0,0]$. These clearly do not share a vertex, and hence $P_3^{2n-1,2n-1}(n-1)$ is not
Eulerian.

Next, suppose $k$ is odd. By Lemma~\ref{lemma:P_successor rule} an edge $(a_0,\dots,a_{n-1})$ of
$P_k^{r,r}(n-1)$ has only one incoming edge and only one outgoing edge in $P_k^{r,r}(n-1)$, namely
$(a_{n-1},a_0,\dots,a_{n-2})$ and $(a_1,\dots,a_{n-1},a_0)$ respectively. It follows that each
connected component of $P_k^{r,r}(n-1)$ is a single necklace, and so we need only show that
$P_k^{r,r}(n-1)$ contains more than one necklace in the remaining cases. Since $P_k^{r,r}(n-1)$ is
non-empty, these are: $r\in\mathbb{Z}$ and $n+2\leq r\leq(k-1)n-2$; and $r\notin\mathbb{Z}$,
$n\geq3$ and $n-1+k/2<r<(k-1)(n-1)+k/2$.

First suppose $r\in\mathbb{Z}$ and $n+2\leq r\leq(k-1)n-2$. If $k=3$ (and so $n\geq4$), let
$\mathbf{b}$ be an $(n-2)$-tuple with entries from $\{1,2\}$ and digit sum $r-3$, which exists
since $n-2\leq r-3\leq2(n-2)$. Then $(\mathbf{b},0,0)$ and $(\mathbf{b},1,2)$ are both edges of
$P_3^{r,r}(n-1)$, since $w_p(0,0)=3=w_p(1,2)$, and they lie in different necklaces since only one
of them contains a zero. If $k\geq5$, then since
\[ [n+2,(k-1)n-2] = [n-2,(k-1)(n-2)]+[4,2k-4], \]
there exist an integer $u$ with $4\leq u\leq 2k-4$ and a zero-free $(n-2)$-tuple $\mathbf{b}$ with
digit sum $r-u$ (where $\mathbf{b}$ is empty if $n=2$). If $u\leq k$, the edges
$(\mathbf{b},1,u-1)$ and $(\mathbf{b},2,u-2)$ both lie in $P_k^{r,r}(n-1)$; if $u>k$, the edges
$(\mathbf{b},k-1,u-k+1)$ and $(\mathbf{b},k-2,u-k+2)$ do. In either case the two edges are
zero-free, so their pseudoweight equals their digit sum, and they have different multisets of
entries (since $u\neq3$ and $u\neq 2k-3$), so they lie in different necklaces.

Now suppose $r\notin\mathbb{Z}$, $n\geq3$ and $n-1+k/2<r<(k-1)(n-1)+k/2$, and let $m=r-k/2$, so
that $m\in\mathbb{Z}$ and $n-1<m<(k-1)(n-1)$. There are at least two distinct zero-free
$(n-1)$-tuples with digit sum $m$: if $\mathbf{c}$ is one such tuple whose entries are not all
equal, then a non-trivial cyclic shift of $\mathbf{c}$ is another; otherwise $\mathbf{c}=x^{n-1}$
with $1<x<k-1$, and then $(x-1,x+1,x^{n-3})$ is another. If $\mathbf{c}$ and $\mathbf{c}'$ are two
such tuples, then $(0,\mathbf{c})$ and $(0,\mathbf{c}')$ are edges of $P_k^{r,r}(n-1)$. Since each
contains exactly one zero, each of their necklaces contains exactly one $n$-tuple beginning with
$0$, and hence they lie in different necklaces.

In every case $P_k^{r,r}(n-1)$ contains more than one necklace, and hence is not strongly
connected.

Finally, suppose $k$ is even and $n+1< r< (k-1)n-1$.  Suppose $r=n+s(k-2)+t-1$, where $1\leq t<k-1$
and $0\leq s<n$.  Then there exists an edge in $P_k^{r,r}(n-1)$ of the form
$((k-1)^s,t,1^{n-s-1})$. No vertex in this necklace will have more than one incoming edge or
outgoing edge in $P_k^{r,r}(n-1)$ except in the case $t=k/2$.  If $t=k/2$ then the necklace
$[(k-1)^s,0,1^{n-s-1}]$ will share a vertex with the necklace $[(k-1)^s,t,1^{n-s-1}]$ but neither
necklace will share a vertex with any other edges in $P_k^{r,r}(n-1)$.

We next show that $P_k^{r,r}(n-1)$ will always contain another necklace not equal to either
$[(k-1)^s,t,1^{n-s-1}]$ (or $[(k-1)^s,0,1^{n-s-1}]$ if $t=k/2$), establishing that $P_k^{r,r}(n-1)$
is not strongly connected. We need to consider three cases.
\begin{itemize}
\item First suppose $s=0$, i.e.\ there exists an edge in $P_k^{r,r}(n-1)$ of the form
    $(t,1^{n-1})$, where $t>2$ since $r>n+1$.  Then there will also exist an edge of the form
    $(t-1,2,1^{n-2})$ in $P_k^{r,r}(n-1)$, since $n\geq2$, i.e.\ another necklace.
\item Next suppose $s>0$ and $t<k-2$; then, since $k\geq4$, there exists a further edge of the
    form $((k-1)^{s-1},k-2,t+1,1^{n-s-1})$ in $P_k^{r,r}(n-1)$.
\item Finally suppose $s>0$ and $t=k-2$.  Now, since $r< (k-1)n-1$, $n-s-1>0$.  Then the edge
    mentioned above takes the form $((k-1)^s,k-2,1^{n-s-1})$.  An additional edge is then
    $((k-1)^s,k-3,2,1^{n-s-2})$.
\end{itemize}
The result follows.
\end{proof}

The $k$ even case is somewhat simpler than the $k$ odd case, since if $k$ is even then
$w_p(\mathbf{a})\in\mathbb{Z}$ for any $k$-ary $n$-tuple $\mathbf{a}$.  We first need the following
preliminary lemma.

\begin{lemma}  \label{lemma:change_zeros_to_half_k}
Suppose $n\geq2$, $k>2$ is even, and $n\leq r\leq n(k-1)$. Suppose $\mathbf{a}$ is an edge (i.e.\ a
$k$-ary $n$-tuple) in $P_k^{r,r}(n-1)$ and $\bar{\mathbf{a}}$ is equal to $\mathbf{a}$ with every
$0$ changed to $k/2$. Then there exist directed paths in $P_k^{r,r}(n-1)$ from $\mathbf{a}$ to
$\bar{\mathbf{a}}$ and from $\bar{\mathbf{a}}$ to $\mathbf{a}$.
\end{lemma}

\begin{proof}
The result follows using repeated applications of Lemma~\ref{lemma:one_change}.  That is, to obtain
a directed path from $\mathbf{a}$ to $\bar{\mathbf{a}}$, simply use Lemma~\ref{lemma:one_change} to
obtain a directed path from $\mathbf{a}$ to an edge $\mathbf{a}'$ equal to $\mathbf{a}$ with a
single $0$ changed to $k/2$, and then a directed path from $\mathbf{a}'$ to an edge $\mathbf{a}''$
equal to $\mathbf{a}'$ with another $0$ changed to $k/2$, and so on until $\bar{\mathbf{a}}$ is
reached.  For the path in the opposite direction simply change the relevant entries of $k/2$ to $0$
one by one until $\mathbf{a}$ is reached.
\end{proof}

We can now give the main result.

\begin{theorem} \label{theorem:P_characteristic_k_even}
Suppose $n\geq2$ and $k>2$ is even. Then $P_k^{r,r+1}(n-1)$ is Eulerian for every $r\in \mathbb{Z}$
satisfying $n\leq r<(k-1)n$.
\end{theorem}

\begin{proof}
By Lemma~\ref{lemma:balanced} we need only show that $P_k^{r,r+1}(n-1)$ is strongly connected. By
Lemma~\ref{lemma:change_zeros_to_half_k}, every edge of $P_k^{r,r+1}(n-1)$ is joined in both
directions by directed paths in $P_k^{r,r+1}(n-1)$ to an edge containing no zeros. Hence we need
only show that there is a directed path in $P_k^{r,r+1}(n-1)$ between any two edges whose
corresponding $n$-tuples contain no zeros.

Let $Z$ be the subgraph of $P_k^{r,r+1}(n-1)$ whose edges are the zero-free $n$-tuples in
$P_k^{r,r+1}(n-1)$. If $\mathbf{a}=(a_0,a_1,\dots,a_{n-1})$ is zero-free then
$w_p(\mathbf{a})=\sum_{i=0}^{n-1}a_i$, and so the map
\[ (a_0,a_1,\dots,a_{n-1}) \mapsto (a_0-1,a_1-1,\dots,a_{n-1}-1) \]
is a bijection from the edges of $Z$ to the $(k-1)$-ary $n$-tuples with digit sum $r-n$ or $r-n+1$.
Since it is induced by a bijection on symbols, it preserves the de Bruijn adjacency, and so it is
an isomorphism from $Z$ to $S_{k-1}^{r-n,r-n+1}(n-1)$. Moreover $0\leq r-n<(k-2)n$, which is
precisely the range of $r$ required for $S_{k-1}^{r,r+1}(n-1)$.

If $k>4$ then $k-1>2$, and $S_{k-1}^{r-n,r-n+1}(n-1)$ is strongly connected by
Theorem~\ref{theorem:S_characteristic}(ii). If $k=4$ then $k-1=2$ and so the digit sum of a binary
$n$-tuple equals its Hamming weight, so $S_2^{r-n,r-n+1}(n-1)=H_2^{r-n,r-n+1}(n-1)$, which is
strongly connected by Theorem~\ref{theorem:Hamming_characterisation}(ii). In either case $Z$ is
strongly connected, and the result follows.
\end{proof}

To analyse the $k$ odd case we first need the following.

\begin{lemma}  \label{lemma:k_odd_zero-free_path}
Suppose $n\geq2$, $k>2$ is odd, and $n\leq r \leq n(k-1)$, where $r=t/2$ for some $t\in\mathbb{Z}$.
Then:
\begin{itemize}
\item[(i)] if $\mathbf{a}$ has pseudoweight $r$, there exists a directed path in
    $P_k^{r,r+0.5}(n-1)$ from $\mathbf{a}$ to $\mathbf{a}'$, where $\mathbf{a}'$ contains no
    entries equal to 0;
\item[(ii)] if $\mathbf{b}$ has pseudoweight $r+0.5$, there exists a directed path in
    $P_k^{r,r+0.5}(n-1)$ from $\mathbf{b}'$ to $\mathbf{b}$, where $\mathbf{b}'$ contains no
    entries equal to 0.
\end{itemize}
\end{lemma}

\begin{proof}
\begin{itemize}
\item[(i)]  First observe that, for any $k$-ary $n$-tuple $\mathbf{a}$ of pseudoweight $s$,
    replacing a single zero in $\mathbf{a}$ by $(k-1)/2$ results in a $k$-ary $n$-tuple of
    pseudoweight $s-0.5$, and replacing a single zero by $(k+1)/2$ results in a $k$-ary
    $n$-tuple of pseudoweight $s+0.5$. So, starting from $\mathbf{a}$ (of pseudoweight $r$),
    replace every $0$ one by one with alternately $(k+1)/2$ and $(k-1)/2$ until an $n$-tuple is
    reached (which we label $\mathbf{a}'$) containing no $0$ entries. In every case the
    pseudoweight of the result will be either $r$ or $r+0.5$, and hence by
    Lemma~\ref{lemma:one_change} there is a directed path from the starting $n$-tuple to the
    modified $n$-tuple in $P_k^{r,r+0.5}(n-1)$.  Joining these paths together yields the
    desired directed path from $\mathbf{a}$ to $\mathbf{a}'$.
\item[(ii)] The argument for (i) works in reverse, replacing every $0$ entry in
    $\mathbf{b}$ with alternately $(k-1)/2$ and $(k+1)/2$, yielding a directed path in
    $P_k^{r,r+0.5}(n-1)$ starting at an $n$-tuple $\mathbf{b}'$ containing no entries equal
    to $0$ and ending in $\mathbf{b}$.
\end{itemize}
\end{proof}

\begin{remark}
Of course, if the pseudoweight of an $n$-tuple is less than $n+(k-2)/2$ or more than
$n(k-1)-(k-2)/2$ then an $n$-tuple cannot contain any zeros, so the Lemma holds trivially.
\end{remark}

We can now give a companion result to Theorem~\ref{theorem:P_characteristic_k_even}.

\begin{theorem} \label{theorem:P_characteristic_k_odd}
Suppose $n\geq2$ and $k>2$ is odd. Then $P_k^{r,r+1}(n-1)$ is Eulerian for every $r\in \mathbb{Z}$
satisfying $n\leq r<(k-1)n$.
\end{theorem}

\begin{proof}
As previously, by Lemma~\ref{lemma:balanced} we need only show that $P_k^{r,r+1}(n-1)$ is strongly
connected. Since $k$ is odd, every edge $\mathbf{a}$ of $P_k^{r,r+1}(n-1)$ satisfies
$w_p(\mathbf{a})\in\{r,r+0.5,r+1\}$. Note also that a zero-free $n$-tuple has integer pseudoweight,
so any zero-free edge of $P_k^{r,r+1}(n-1)$ has pseudoweight $r$ or $r+1$.

We first show that from every edge $\mathbf{a}$ of $P_k^{r,r+1}(n-1)$ there is a directed path in
$P_k^{r,r+1}(n-1)$ to a zero-free edge $\mathbf{a}^*$. If $\mathbf{a}$ is zero-free, set
$\mathbf{a}^*=\mathbf{a}$. Otherwise there are three cases.
\begin{itemize}
\item If $w_p(\mathbf{a})=r$, then Lemma~\ref{lemma:k_odd_zero-free_path}(i) gives the required
    path in $P_k^{r,r+0.5}(n-1)\subseteq P_k^{r,r+1}(n-1)$.
\item If $w_p(\mathbf{a})=r+0.5$, then Lemma~\ref{lemma:k_odd_zero-free_path}(i), applied with
    $r+0.5$ in place of $r$, gives the required path in $P_k^{r+0.5,r+1}(n-1)\subseteq
    P_k^{r,r+1}(n-1)$.
\item If $w_p(\mathbf{a})=r+1$, let $\mathbf{a}'$ equal $\mathbf{a}$ with one zero replaced by
    $(k-1)/2$, so that $w_p(\mathbf{a}')=r+0.5$. Since $\mathbf{a}$ and $\mathbf{a}'$ differ in
    a single position, Lemma~\ref{lemma:one_change}(i) gives a directed path from $\mathbf{a}$
    to $\mathbf{a}'$ in $P_k^{r+0.5,r+1}(n-1)$. The previous case then gives a directed path
    from $\mathbf{a}'$ to a zero-free edge $\mathbf{a}^*$.
\end{itemize}

Symmetrically, we show that for every edge $\mathbf{b}$ of $P_k^{r,r+1}(n-1)$ there is a directed
path in $P_k^{r,r+1}(n-1)$ from a zero-free edge $\mathbf{b}^*$ to $\mathbf{b}$. If $\mathbf{b}$ is
zero-free, set $\mathbf{b}^*=\mathbf{b}$. Otherwise there are again three cases.
\begin{itemize}
\item If $w_p(\mathbf{b})=r+1$, then Lemma~\ref{lemma:k_odd_zero-free_path}(ii), applied with
    $r+0.5$ in place of $r$, gives the required path in $P_k^{r+0.5,r+1}(n-1)$.
\item If $w_p(\mathbf{b})=r+0.5$, then Lemma~\ref{lemma:k_odd_zero-free_path}(ii) gives the
    required path in $P_k^{r,r+0.5}(n-1)$.
\item If $w_p(\mathbf{b})=r$, let $\mathbf{b}'$ equal $\mathbf{b}$ with one zero replaced by
    $(k+1)/2$, so that $w_p(\mathbf{b}')=r+0.5$. Lemma~\ref{lemma:one_change}(i) gives a
    directed path from $\mathbf{b}'$ to $\mathbf{b}$ in $P_k^{r,r+0.5}(n-1)$. The previous case
    then gives a directed path from a zero-free edge $\mathbf{b}^*$ to $\mathbf{b}'$.
\end{itemize}

It remains to show that there is a directed path in $P_k^{r,r+1}(n-1)$ from $\mathbf{a}^*$ to
$\mathbf{b}^*$. Let $Z$ be the subgraph of $P_k^{r,r+1}(n-1)$ whose edges are its zero-free
$n$-tuples. Exactly as in the proof of Theorem~\ref{theorem:P_characteristic_k_even}, subtracting
one from every entry gives an isomorphism from $Z$ to $S_{k-1}^{r-n,r-n+1}(n-1)$, where $0\leq
r-n<(k-2)n$. If $k>3$ then $k-1>2$, and this graph is strongly connected by
Theorem~\ref{theorem:S_characteristic}(ii). If $k=3$ then $k-1=2$, and
$S_2^{r-n,r-n+1}(n-1)=H_2^{r-n,r-n+1}(n-1)$ is strongly connected by
Theorem~\ref{theorem:Hamming_characterisation}(ii). Hence $Z$ is strongly connected, so there is a
directed path from $\mathbf{a}^*$ to $\mathbf{b}^*$, and the result follows.
\end{proof}

The separate results for $k$ odd and $k$ even (Theorems~\ref{theorem:P_characteristic_k_even} and
\ref{theorem:P_characteristic_k_odd}) can be combined to give the following, whose proof is
immediate using the same arguments as for Corollary~\ref{corollary:H_characterisation}.

\begin{corollary} \label{corollary:P_characteristic}
Suppose $n\geq2$, $k>2$ and $n\leq r< s\leq (k-1)n$, where $r,s\in\mathbb{Z}$. Then
$P_k^{r,s}(n-1)$ is Eulerian.
\end{corollary}

\begin{remark}
The requirement that $r,s\in\mathbb{Z}$ in Corollary~\ref{corollary:P_characteristic} cannot in
general be dropped when $k\geq5$ is odd. For example, $P_5^{9/2,11/2}(1)$ is not strongly
connected, since its edges $(1,4)$ and $(4,1)$ form a component on the vertices $\{1,4\}$ that is
disjoint from its other edges.
\end{remark}

\begin{example}  \label{example:pseudoweight}
We give examples of cut-down de Bruijn sequences arising from Eulerian circuits in subgraphs
$P_k^{r,s}(n-1)$; in each case one period of the sequence is shown. Recall that $f(0)=k/2$, so that
$0$ and $k/2$ contribute equally to the pseudoweight when $k$ is even.
\begin{itemize}
\item[(i)] Let $k=4$, $n=3$ and $r=s=4=n+1$. By Theorem~\ref{theorem:P_characteristic_a}(ii),
    $P_4^{4,4}(2)$ is Eulerian. Its six edges form the necklaces $[0,1,1]$ and $[1,1,2]$, and
    the sequence
    \[ (121101) \]
    of period 6 contains each of these 3-tuples exactly once in a period.
\item[(ii)] For $k=3$ the subgraph $P_3^{4,4}(2)$ is not Eulerian, as
    Theorem~\ref{theorem:P_characteristic_a} requires $k\neq3$ in case (ii). Here $f(0)=3/2$,
    and the six edges of $P_3^{4,4}(2)$ form the necklaces $[0,0,1]$ and $[1,1,2]$, which are
    disjoint directed cycles in $B_3(2)$.
\item[(iii)] Let $k=4$, $n=3$, $r=4$ and $s=5$. By
    Theorem~\ref{theorem:P_characteristic_k_even}, $P_4^{4,5}(2)$ is Eulerian; it has $6+15=21$
    edges, and the sequence
    \[ (101102113112012122100) \]
    of period 21 contains every 4-ary 3-tuple of pseudoweight 4 or 5 exactly once in a period.
\item[(iv)] Let $k=3$, $n=3$, $r=4$ and $s=5$. By Theorem~\ref{theorem:P_characteristic_k_odd},
    $P_3^{4,5}(2)$ is Eulerian. Since $k$ is odd it contains 3-tuples of pseudoweight $4$,
    $9/2$ and $5$; there are 6, 7 and 6 of these respectively (note that $w_p(0,0,0)=9/2$). The
    sequence
    \[ (0100201210211221200) \]
    of period 19 contains every ternary 3-tuple of pseudoweight between 4 and 5 exactly once in
    a period.
\end{itemize}
\end{example}

\section{Concluding remarks}

Perhaps the most pressing question is how to efficiently generate the universal sequences which are
shown to exist in this paper.  The results given here in every case show that certain subgraphs of
the de Bruijn graph are Eulerian, and hence there exists a universal sequence containing all the
$n$-tuples corresponding to edges in the relevant subgraph.  However this does not, in itself, show
how to construct an example of a universal sequence.  Of course, as noted in
Section~\ref{section:introduction}, there are simple and relatively efficient algorithms for
finding Eulerian circuits --- and hence in our case a universal sequence --- in any Eulerian graph.

However, for certain applications, the standard algorithms for finding circuits might not be
appropriate.  The issue of finding very efficient algorithms for generating circuits in specific
subgraphs of the de Bruijn graph has been addressed by a number of authors.
\begin{itemize}
\item In 2013 Sawada et al.~\cite{Sawada13} described a method of efficiently constructing
    universal sequences containing all binary $n$-tuples with Hamming weight in the range
    $[c,d]$ for any $c,d$ satisfying $0\leq c<d\leq n$; a further result of this type is given
    in the 2014 paper of Sawada et al.~\cite{Sawada14}.
\item Recently, in 2026, Sawada et al.~\cite{Sawada26} gave an approach using `concatenation
    trees' to construct a range of classes of universal sequences in a highly efficient way,
    including orientable sequences.
\item Again in 2026, Campbell et al.~\cite{Campbell26} showed how to efficiently construct
    universal sequences for $t$-subsets and $t$-multisets.
\item Finally, also in 2026, in related work, Gabri\'{c} et al.~\cite{Gabric26} described how
    to efficiently decode (i.e.\ find the location of any given $n$-tuple) certain classes of
    universal sequences for $t$-subsets and $t$-multisets.
\end{itemize}

Of particular relevance to the work here are the results of Sawada et al.~\cite{Sawada16}. They
show (\cite[Theorem 1]{Sawada16}) that the (well known and very efficient) greedy and FKM
algorithms create the lexicographically smallest universal cycle for any subgraph whose set of
edges is in $\mathcal{C}(n,k)$ --- see Section~\ref{subsection:Sawada-Williams-Wong}.  As we
observed in Remarks~\ref{remark:H_SWW}, \ref{remark:S_SWW} and \ref{remark:P_SWW}, any set
containing all $n$-tuples of weight at least a specified lower bound is in $\mathcal{C}(n,k)$ for
all three of the weight functions considered here.  This immediately yields efficient methods of
construction for these classes of universal sequence.  While this by no means solves the problem of
efficiently constructing universal sequences for all the cases we address here, it is at least a
partial result.  We leave the problem of finding such efficient methods of construction for all the
cases considered in this paper for future work.

It would also be of interest to consider whether similar results could be obtained for other
possible permutation-agnostic weight functions. Result~\ref{result:Sawada-Williams-Wong} may well
be relevant for other possible weight functions.

\section*{Acknowledgements}

We would like to sincerely thank the referees for their careful checking and extremely valuable
comments, without which the paper would be much the poorer.  We must also thank the Claude LLM for
spotting a range of minor errors, and also for helping with the examples.


\providecommand{\bysame}{\leavevmode\hbox to3em{\hrulefill}\thinspace}
\providecommand{\MR}{\relax\ifhmode\unskip\space\fi MR }
\providecommand{\MRhref}[2]{%
  \href{http://www.ams.org/mathscinet-getitem?mr=#1}{#2}
} \providecommand{\href}[2]{#2}

\end{document}